%% file: revised_converse_landing_V3_for_ArXiv.tex
\documentclass[12pt]{article}

\usepackage[english]{babel}

\usepackage[T1]{fontenc}
\usepackage{graphicx}
\usepackage[all]{xy}
\usepackage[pagewise]{lineno}

\usepackage{cite}
\usepackage{amsmath}
\usepackage{amscd}
\usepackage{amsfonts}
\usepackage{amssymb}
\usepackage{amsthm}
\usepackage{parskip}
\usepackage{bm}
\usepackage{listings}
\usepackage{pifont}
\usepackage{authblk}
\usepackage{epsfig}
\usepackage{psfrag}
\usepackage[margin=1.4 in]{geometry}

\usepackage{latexsym}
\usepackage[usenames]{color}
\usepackage[linktocpage=true]{hyperref}

\theoremstyle{plain}

\newtheorem*{Main Theorem*}{Main Theorem}
\newtheorem{thm}{Theorem}
\numberwithin{thm}{section}
\newtheorem{prop}[thm]{Proposition}
\newtheorem{lem}[thm]{Lemma}

\newtheorem{defn}[thm]{Definition}
\newtheorem{rem}[thm]{Remark}
\newtheorem*{Corollary of the Main Theorem*}{Corollary of the Main Theorem}

\theoremstyle{definition}
\renewcommand\Im{\operatorname{Im}}
\renewcommand\Re{\operatorname{Re}}

\begin{document}

\title{A converse landing theorem in parameter spaces}

\author{Asl{\i} Deniz}
\affil{Department of Mathematics and Statistics \\College of Business Administration \\American University of the Middle East, Kuwait\\
{seaslideniz@gmail.com}}

\maketitle

\begin{abstract}
In this article, we prove that for several one-dimensional holomorphic families of holomorphic maps, in the parameter plane, there exists a local piece of a curve that lands at a given parabolic parameter, in the spirit of well-known results about the quadratic and the exponential families. We also show that, under some assumptions, this general result partially answers the existence and landing questions of ray structures in the parameter planes for holomorphic families of transcendental entire maps.
\end{abstract}

\begin{itemize}
\item Keywords: holomorphic dynamics, parameter rays, landing theorem
\item 2010 Mathematics Subject Classification: 37F45, 30D05, 30D20
\end{itemize}
\section{Introduction}

To investigate dynamical and parameter planes for holomorphic maps, one of the main strategies is to study dynamic and parameter rays, respectively, and their landing behaviours.
The concept of the ray is first developed in \cite{DH} in the study of the quadratic family $Q_c(z)=z^2+c$: the Julia set $J_c$ is the boundary of the filled-in Julia set $K_c$ -- the points with bounded orbits. The orbit of the critical point $z=0$ has an important role in the topology of $J_c$ and $K_c$: both $K_c$ and $J_c$ are connected if and only if the critical point  has a bounded orbit. When $K_c$ is connected, using the B\"{o}ttcher Coordinate $\psi_c:\mathbb{C}\backslash K_c\rightarrow\mathbb{C}\backslash \overline{\mathbb{D}}$ where $\mathbb{D}$ is the unit disk, the external ray of argument $\theta$ is defined as the inverse image $\psi_c^{-1}\{re^{2\pi i\theta}, r>1\}$. The periodicity of an external ray is then explained by the periodicity of its argument. An external ray of argument $\theta$ is said to land if $\lim_{r\rightarrow 1}\psi_c^{-1}\{re^{2\pi i\theta}, r>1\}$ exists. Carath\'eodory-Torhosrt Theorem says that all external rays land and the landing point depends continuously on the argument if and only if $K_c$ is locally connected (see \cite[Theorem 18.3]{Mil}). Douady and Hubbard presented their famous result in landing questions, regardless of local connectivity: If $K_c$ is connected, all periodic external rays land at a repelling or parabolic periodic point, and conversely, every repelling and parabolic periodic point is the landing point of periodic external rays \cite{DH}.

This work has been a great motivation for studying ray structures and their landing properties for transcendental maps. Such investigations began with the exponential family $E_{\kappa}(z)=e^{z}+\kappa$, the analogue of the quadratic family in terms of having only one singular value (but this time an asymptotic value). Similar to the external rays for the quadratic maps, it has been observed that the path-connected components of the escaping set form curves which are called dynamic rays (hairs). The dynamic ray structure for exponential maps has been studied more than a hundred years back in \cite{F} and later in \cite{BDGH,DK,DT,SZ1,SZ2}. Later on, the existence of dynamic rays for many more complicated transcendental maps has been proven in \cite{B2007,BaXaRe,RRRS, NuriaDavid}. There are many landing results for many types of transcendental entire maps, such as \cite{AnnaNuria,LyubichBenini,RemBen,Deniz,Helena,RemSiegel,Rem2006}. Of such works, \cite{RemBen} is the most significant in the sense that it gives a perfect correspondence to Douady and Hubbard's landing theorem: If a transcendental entire map with dynamic rays has bounded postsingular set, then all dynamic rays land, and conversely, all parabolic and repelling periodic points are landing points of some periodic rays.

Similar to the analysis in dynamical planes, the study of one-dimensional parameter spaces can be considered in terms of the study of curves consisting of escaping parameters and in terms of their landing behaviours. These curves are called parameter rays. By an escaping parameter, we mean a parameter for which one of the singular values escapes to infinity along dynamic rays with a prescribed identification. According to the leading work by Douady and Hubbard in the quadratic family, the Riemann map $\Psi:\mathbb{C}\backslash \mathcal{M}\rightarrow\mathbb{C}\backslash\overline{\mathbb{D}}$ is constructed such that $\Psi(c)=\psi_c(c)$, where $\psi_c$ is the corresponding B\"{o}ttcher coordinate and $\mathcal{M}$ is the Mandelbrot set. Hence, a parameter ray of argument $\theta$ is defined as $\Psi_c^{-1}\{re^{2\pi i\theta}, r>1\}$. With this construction, for a parameter on a ray of some argument, in the corresponding dynamical plane, the critical value is on the dynamic ray of the same argument. A parameter ray of argument $\theta$ is said to land if $\lim_{r\rightarrow 1}\Psi_c^{-1}\{re^{2\pi i\theta}, r>1\}$ exists. The landing results in \cite{DH} for parameter rays of periodic arguments also have two directions: every parameter ray of a periodic argument lands at a parabolic parameter $c_0$. Moreover, in the dynamical plane of $c_0$, the dynamic ray of argument $\theta$ lands at a parabolic periodic point. Conversely, every parabolic parameter is the landing point of parameter rays of periodic arguments. The same result is established in \cite{SchMandelbrotset} using combinatorial arguments. As an important result in the same spirit, a landing theorem for parameter rays of rational arguments is presented in \cite{CarstenRyd} of the connectedness loci for the one parameter families of polynomials $P_c(z)=z^{d}+c$ where $d\geq 2$. 
Although computer simulations suggest that escaping parameters form curves in parameter spaces for one-dimensional families of transcendental maps, there is no complete answer to the question of existence of parameter rays, except for the exponential family. The first study of the existence of parameter rays in the exponential family was carried out by Devaney and coauthors in \cite{BDGH}. Schleicher worked broadly on combinatorial analysis of the exponential family in his habilitation thesis \cite{sch1999}, and the construction of the parameter rays was completed in \cite{for2003, forsch2009, RemSch2}. In \cite{sch1999}, the landing properties of the parameter rays are also discussed. More precisely, it is shown that every parabolic parameter is the landing point of one or two parameter rays.
 
 Moreover, based on this result, it is shown that every parameter ray at periodic address lands, which then appeared also in \cite{RemLasBifurcation}. A very recent study by Astorg and coauthors \cite{AsBeFa} in natural families of rational, entire or meromorphic functions of finite type has some partial results on this problem: Roughly they show that under some mild conditions, a singular parameter is the endpoint of a curve of parameters for which an attracting cycle exits the domain, while the multiplier converges to zero.

  Here, we would like to distinguish two main questions in the study of rays in both dynamical and parameter planes: whether a ray lands or there exists a ray landing at a considered point. We use terms ``landing" and ``converse landing" for these two questions, respectively. In this paper, we prove a converse landing theorem in some one-dimensional holomorphic families of holomorphic maps. In the Main Theorem we show that, under some assumptions, for a parabolic parameter $a_0$, if a forward invariant curve lands at at parabolic fixed point in the dynamical plane, then in the parameter plane there exists a curve piece  which lands at $a_0$, and which consists of parameters for which a distinguished singular value is on a forward invariant curve. Our work not only answers a landing question but also answers (partially) the existence question. Since we use a local theory, we prove the existence of only a piece of curve in parameter plane, not a full length. On purpose, we do not mention rays, so that our result can apply to many cases as long as local conditions hold regardless of whether dynamic and parameter rays are constructed, or proven to exist. To the best of our knowledge regarding the existence of parameter rays and their landing behaviours, nothing has been proven for families of transcendental entire maps in a general setting. Our result serves also a partial answers to these problems, which is separately given by Corollary of the Main Theorem.

\subsection{Theorems, basic definitions and notations}

The dynamical plane for a holomorphic map is partitioned into two totally invariant sets with respect to the behaviour of the points: the set of points with stable behaviour and its complement, which are called the \textit{Fatou set} and the \textit{Julia set}, respectively.
If the orbit of a point $z_0$ under a holomorphic map $f$ is finite, we say that $z_0$ is pre-periodic. If there exists $q\in\mathbb{N}$ satisfying $f^q(z_0)=z_0$, then $z_0$ is periodic; in particular, $z_0$ is a fixed point if $q=1$. Periodic points are classified in terms of their multipliers $\mu(z_0):=(f^q)'(z_0)$:
If $|\mu(z_0)|<1$, $z_0$ is an attracting, if $|\mu(z_0)|>1$, $z_0$ is a repelling, and if $|\mu(z_0)|=1$, $z_0$ is an indifferent periodic point. If $\mu(z_0)=e^{2\pi i \theta}$, where $\theta\in\mathbb{Q}$, then $z_0$ is called a parabolic periodic point.
Singular values are points that have no neighbourhood where all inverse branches of $f$ are well defined and univalent. 
A critical value is the image of a critical point. An asymptotic value is a point $a\in\mathbb{C}$ for which there exists a curve tending to infinity while its image under $f$ converges at $a$. The singular set of a function $f$ is the union of critical and asymptotic values and their accumulation points.

\begin{defn}[Holomorphic family]
Let $\mathcal{A}$ be a complex manifold. A map 
\begin{eqnarray*}
\mathcal{F}:\mathcal{A}\times\mathbb{C}&\rightarrow&\widehat{\mathbb{C}}\\
(\textbf{a}, z)&\mapsto& f_{\textbf{a}}(z)
\end{eqnarray*}
is called a holomorphic family of holomorphic maps over the parameter space $\mathcal{A}$ if

\begin{itemize}
    \item $\textbf{a}\mapsto f_{\textbf{a}}(z)$ is holomorphic for all $z\in\mathbb{C}$,
    \item $z\mapsto f_{\textbf{a}}(z)$ is holomorphic for all $\textbf{a}\in\mathcal{A}$.
\end{itemize}

\end{defn}
The dimension of $\mathcal{A}$ can be finite, or infinite, related to the singular set of each $f_{\textbf{a}}$. For example, it is known that every entire map with finite number of singular values belongs to a finite dimensional complex manifold % Moreover, for such maps, all singular values are isolated. 
 \cite{EL}. A classical approach to study holomorphic families is to fix all but one of the isolated singular values, and obtain a "one dimensional slice" of the parameter space.
 For a study of specific slices of maps with finite singular set, see for example, \cite{NuriaLinda}. Our setting is rather general, which allows the members of the holomorphic family to have infinitely many singular values, or a dense singular set as well. Also, more than one singular values can bifurcate at the same time. We are interested in bifurcations of only one distinguished isolated singular value, as the only singular value inside some certain Euclidean disk in the corresponding dynamical planes. And other singular values are allowed to vary as long as they are outside that given Euclidean disk. 

\begin{defn}[local natural family with parabolic assumptions]

 Let $\mathcal{F}$ be a holomorphic family of holomorphic maps over $\mathcal{A}$, let $\mathbb{D}(a_0, r)\subset\mathcal{A}$ be a Euclidean disk. We call the restriction $\{f_a\}_{a\in\mathbb{D}(a_0, r)}$ a local natural family with parabolic assumptions in $\mathcal{F}$ if the following conditions hold:

\begin{enumerate}
    
    \item [1.] $a_0$ is a parabolic parameter, and $f_{a_0}$ has a nonpersistent parabolic fixed point $z_0$ satisfying $f_{a_0}'(z_0)=1$ with $f_{a_0}''(z_0)\neq 0$.
    
    \item [2.] There exists $R>0$ such that there exists only one singular value $s(a)$ of $f_a$  in $\overline{\mathbb{D}(z_0,R)}$ for all $a\in\mathbb{D}(a_0,r)$. Moreover
    \begin{itemize}
        \item[i.] $a\mapsto s(a)$ is holomorphic in $\mathbb{D}(a_0,r)$,
        \item[ii.] There exists an attracting petal attached to $z_0$ where $s(a_0)$ is the only singular value.
   \end{itemize}
  
\end{enumerate}

\end{defn}

\begin{defn}[Equivariant holomorphic motion]
A holomorphic motion of $E$ over a domain $M\in\mathbb{C}$ with base point $a_0\in M$ is a map $H:M\times E\rightarrow\widehat{\mathbb{C}}$ which satisfies the following properties:
\begin{itemize}
    \item [i.] $H(a_0, z)=z$,
    \item [ii.] $z\mapsto H(\lambda, z)$ is injective for all $a\in M$,
    \item [iii.] $a\mapsto H(a, z)$ is holomorphic for all $z\in E$.
\end{itemize}
Given $\{f_{a}\}_{a\in M}$, we say that the holomorphic motion $H:M\times E\rightarrow\mathbb{C}$ is equivariant if and only if
\begin{equation*}
f_a(H(a,z))=H(a, f_{a_0}(z))
\end{equation*}
whenever $z$ and $f_{a_0}(z)$ belong to $E$, and for all $a\in M$.

\end{defn}

A curve $\gamma:(-\infty,\mathcal{T})\rightarrow\mathbb{C}$, $\mathcal{T}\in\mathbb{R}\cup\{\infty\}$, $t\mapsto\gamma(t)$ is said to be \textit{a forward invariant curve} for the map $f$, if $f(\gamma)\subset \gamma$ and we parametrize it such that for $t<\mathcal{T}-1$, $f(\gamma(t))=\gamma(t+1)$. We say that $\gamma$ lands if $\lim_{t\rightarrow-\infty}\gamma(t)$ exists. For ease of the expression, we denote by $\gamma$, the whole curve. We use the notation $\gamma_a$ for a forward invariant curve in the dynamical plane generated by the iterates of the map $f_a$. 

\begin{Main Theorem*}[A converse landing theorem for a local natural family with parabolic assumptions]\label{main2}
Consider a local natural family with parabolic assumptions $\{f_a\}_{a\in\mathbb{D}(a_0, r)}$. Suppose that
 there is a forward invariant curve $\gamma_{a_0}:(-\infty,\mathcal{T})\rightarrow\mathbb{C}$, $t\mapsto\gamma_{a_0}(t)$ with the following properties:

\begin{enumerate}

    \item [i.] $\lim_{t\rightarrow-\infty}\gamma_{a_0}(t)=z_0$.
    \item [ii.] $\gamma_{a_0}$ does not contain any critical point or a singular value.
    \item [iii.] There exists $T>0$, a neighbourhood $\mathbb{D}(a_0,\delta)$ of $a_0$, such that $\delta=\delta(T)$, $\delta<r$, and there exists an equivariant holomorphic motion
\begin{eqnarray*}
H:\mathbb{D}(a_0,\delta)\times\gamma_{a_0}[T,\mathcal{T})&\rightarrow&\mathbb{C}\notag\\
(a,\gamma_{a_0}(t))&\mapsto& \gamma_a(t).
\end{eqnarray*}

\end{enumerate}

Then, for some $t_0\in\mathbb{R}$ with $t_0+1<T$, there exists a simple curve  $\Gamma:(-\infty,t_0]\rightarrow\mathbb{D}(a_0,\delta)$ in the parameter plane, such that for each $a=\Gamma(t)$: $s(a)=\gamma_a(t)$. Moreover, $\lim_{t\rightarrow-\infty}\Gamma(t)=a_0$.

\end{Main Theorem*}

In order to exhibit most generality, we present our converse landing theorem without mentioning dynamic rays and parameter rays explicitly. Wherever dynamic rays are constructed, or proven to exist, then "forward invariant curve" is just replaced by the fixed dynamic ray, and curve piece in parameter plane then corresponds to a parameter ray piece, in particular, a "fixed" parameter ray piece. Although there is no dynamics in parameter spaces so being "fixed" has no sense, we use this term to associate a parameter ray with some certain fixed dynamic ray. In fact, our result covers beyond the dynamic ray in the sense of curves in the escaping set, and in relation, the parameter ray in the sense of curves consisting of escaping parameters. For example, as demonstrated in \cite{DenizCon}, dynamic rays can be constructed inside a super-attracting basin (and in this case they are called internal dynamic rays), and parameter rays can be constructed inside a hyperbolic component (and in this case they are called internal parameter rays).  We keep the statement of our theorem so general so that it can also deal with these and many other cases.
\\
\\
Our proof relies on having an equivariant holomorphic motion of the forward invariant curve in consideration for large potentials. 
For polynomials of degree $d\geq 2$, existence of such holomorphic motions are guaranteed by the B\"{o}ttcher's Theorem, so our theorem applies to all polynomials of degree $d\geq 2$. 
\\
\\
Because the landing and the existence of parameter rays for families of transcendental entire maps have been interesting research topics, we present here a partial answer to this question, as an application of our theorem. In families of transcendental entire maps, holomorphic motion is related to  \textit{quasiconformal equivalence}.  Recall that class $\mathcal{B}$, also called the Eremenko-Lyubich class consists of all transcendental entire maps with bounded singular sets. In \cite{remrig}, it is proven that quasiconformally equivalent maps  in class $\mathcal{B}$ are quasiconformally conjugate in a subset of the Julia set. 
This establishes the existence of the required holomorphic motion for maps in class $\mathcal{B}$. 
Here we include the related result.
\begin{prop}{\cite[Proposition 3.6]{remrig}}
Let $f\in\mathcal{B}$, $M$ be a finite-dimensional complex manifold, with base point $a_0\in M$. Suppose that $\{f_{a}\}_{a\in M}$ is a family of transcendental entire maps which are quasiconformally equivalent to $f_{a_0}$ with the equivalences given by $\psi_{a}\circ f_a=f_{a_0}\circ\phi_{a}$, where $\psi_{a_0}=\phi_{a_0}=Id$, and $\psi_{a}$ and  $\phi_{a}$ depends on the parameter $a$ holomorphically.
Let $N$ be a  compact subset of $M$, with $a_0\in N$. Then there exists a constant $R>0$ and a set
\begin{equation*}
    \mathcal{J}_R:=\{z\in\mathbb{C}:\;\;|f_{a_0}^n(z)|>R \;\;\text{for all}\;\; n\geq 1\},
\end{equation*}
such that there exists an equivariant holomorphic motion of $J_R$ over the interior of $N$.
\end{prop}
 
Here we include a definition of dynamic ray, from \cite{AnnaNuria}.
 \begin{defn}[Dynamic ray]
A dynamic ray for a transcendental entire map $f$ is an injective curve $\gamma:(-\infty,\infty)\rightarrow \mathcal{I}(f)$, as a maximal set in the escaping set satisfying the following properties:
\begin{itemize}
    \item[i.] $\lim_{t\rightarrow\infty}|f^n(\gamma(t))|=\infty$, for all $n\geq 0$,
    \item[ii.] $\lim_{n\rightarrow\infty}|f^n(\gamma(t))|=\infty$, uniformly on $[t_0,\infty)$ for all $t_0\geq, -\infty$
    \item[iii.] $\gamma(t)$ is not a critical point for any $t\in\mathbb{R}$.
\end{itemize}

\end{defn}
  We denote by $\widehat{\mathcal{B}}$ the class of transcendental entire maps in class $\mathcal{B}$, of finite order, or finite composition of such maps. For this class, the existence of a dynamic ray structure is proven by Rottenfusser et al. \cite{RRRS}. A dynamic ray $\gamma$ is called \textit{a fixed dynamic ray}  if it is forward invariant.

\begin{defn}[Fixed parameter ray]

A fixed parameter ray associated with the singular value $s(a)$ and a fixed dynamic ray $\gamma_a$ for the local natural family with parabolic assumptions $\{f_a\}_{a\in\mathbb{D}(a_0, r)}\subset\widehat{\mathcal{B}}$ is an injective curve $\Gamma:(-\infty,t_0]\rightarrow\mathbb{D}(a_0, r)$, $t\mapsto \Gamma(t)$ for some $t_0<\infty$ with the property that for $\Gamma(t)=a$, then $s(a)=\gamma_a(t)$.
\end{defn}

\begin{Corollary of the Main Theorem*}[A converse landing theorem for some families of transcendental entire maps]
Consider a local natural family with parabolic assumptions $\{f_a\}_{a\in\mathbb{D}(a_0, r)}$ in $\widehat{\mathcal{B}}$,  of quasiconformally equivalent maps. Suppose that there exists a fixed dynamic ray $\gamma_{a_0}$, which lands at $z_0$ and which does not contain a singular value. Then, for some $t_0\in\mathbb{R}$, there exists a fixed parameter ray piece $\Gamma:(-\infty,t_0]\rightarrow\mathbb{D}(a_0,\delta)$, $\delta<r$. Moreover, $\Gamma(t)$ lands at $a_0$.

\end{Corollary of the Main Theorem*}

In several cases, it has been proven already that under some certain conditions, given parabolic fixed point is the landing point of at least one fixed ray (see  \cite{AnnaNuria, RemBen, Helena}). Our theorem easily can be adapted to those with a little modification.

\subsubsection*{Structure of the paper}
Section~\ref{sectionfatouimplosion} recalls the Fatou coordinates and the parabolic implosion phenomenon. The proof of the Main Theorem requires assigning a new parameter. This is given in Section~\ref{mu}. In Section~\ref{sectionproof}, we prove the Main Theorem by using the new parameter.

\section{Fatou coordinates and parabolic implosion}\label{sectionfatouimplosion}

Here, we summarize the concept of Fatou coordinates and the phenomenon of the parabolic implosion, mostly following \cite{Mil, shis2000}. 

Consider the general form of maps with parabolic fixed point at $0$ with the multiplier $1$:
\begin{equation}\label{generalparabolic}
    f(z)=z+az^{n+1}+...\;\;\;\;\;\;\;\;\;\; n\geq 1 \;\;\text{and}\;\; a\neq 0.
\end{equation}
The number $n+1$ is called the multiplicity of the fixed point $0$. We are interested in the case where $n+1\geq 2$.
\begin{defn}[Attraction and repulsion vectors]
A complex number $\textbf{v}$ is called a repulsion vector for $f$ at the origin if $na\textbf{v}^n=1$, and an attraction vector if $na\textbf{v}^n=-1$.
\end{defn}

Now consider the parabolic fixed point $z_0$ for $f_{a_0}$. Since $f_{a_0}'(z_0)=1$ and $f_{a_0}''(z_0)\neq 0$,  $f_{a_0}$ can be written as

\begin{equation*}
    f_{a_0}(z)=z+\frac{f_{a_0}''(z_0)}{2}(z-z_0)^2+...
\end{equation*}
By a change of coordinate we can assume that the fixed point is $0$, and by another change of coordinate we can assume that $f''_{a_0}(0)=2$, that is,

\begin{equation}
    f_{a_0}(z)=z+z^2+...
\end{equation}
Comparing with the general form (\ref{generalparabolic}), here $n=1$ and $a=1$, so the only repulsion vector is  $<1,0>$ and the only attraction vector is $<-1,0>$ for $f_{a_0}$ at the origin.

\begin{defn}[Attracting and repelling petals]
Let $U$ be a simply connected neighborhood of $z_0$ such that $f_{a_0}:U\rightarrow f_{a_0}(U)$ is univalent. An attracting petal $\Omega^{+}\subset U$ for $f_{a_0}$ for the attraction vector is an open set such that
\begin{enumerate}
    \item [i.] $z_0$ is on the boundary of $\Omega^{+}$,
    \item [ii.] $f_{a_0}(\Omega^{+})\subset \Omega^{+}$,
    \item [iii.] an orbit under $f_{a_0}$ is eventually absorbed by $\Omega^{+}$ if and only if it converges to $z_0$ from the direction of the attraction vector.
\end{enumerate}
Let $g_{a_0}:f_{a_0}(U)\rightarrow U$ be the inverse branch of $f_{a_0}$ which fixes $z_0$. A repelling petal $\Omega^{-}\subset U$ for the repulsion vector is then defined as an attracting petal for $g_{a_0}$.
\end{defn}

Petals and their definitions are not unique. For more details see, for example, \cite{Mil}, \cite{shis2000}.
\begin{thm}[Existence of Fatou coordinates]
There exist a pair of univalent maps, $\phi_{a_0}^{+}:\Omega^{+}\rightarrow \mathbb{C}$ and $\phi_{a_0}^{-}:\Omega^{-}\rightarrow\mathbb{C}$, satisfying
\begin{equation*}
\phi_{a_0}^{+}(f_{a_0}(z))=\phi_{a_0}^{+}(z)+1\;\;\mathrm{and}\;\;\phi_{a_0}^{-}(f_{a_0}(z))=\phi_{a_0}^{-}(z)+1.
\end{equation*}
\end{thm}

The maps $\phi_{a_0}^{+}$ and $\phi_{a_0}^{-}$ are called
\textit{the incoming Fatou coordinate} and \textit{the outgoing Fatou coordinate}, respectively. Let us define $\psi_{a_0}^{+}:=(\phi_{a_0}^{+})^{-1}$, and $\psi_{a_0}^{-}:=(\phi_{a_0}^-)^{-1}$. By using the dynamics, $\phi_{a_0}^{+}$ extends to the whole parabolic basin, and  $\psi_{a_0}^{-}$ extends to the whole complex plane. However, these extensions are no longer univalent because of the presence of singular value(s).

In the dynamical plane for $f_{a_0}$, there are crescent-shaped fundamental domains, each of which is bounded by a pair of curves meeting at $z_0$, where one curve maps to the other under $f_{a_0}$. We are interested in perturbations of the map $f_{a_0}$ to a map $f_a$ nearby. If the parabolic fixed point is nonpersistent, it bifurcates into two fixed points, each of which has a multiplier close to $1$. These fixed points can be attracting, repelling, or indifferent. When the multipliers of the fixed points are not real, the fundamental domains continue to exist. In such cases, the boundary curves of each fundamental domain meet at the two distinct fixed points. A gate opens between these fixed points, and any point goes through the gate under the iteration. This is referred to ``egg beater" dynamics (see Figure~\ref{eggbeatereps}). In this case, conjugation between the dynamics and the translation $T_1:z\mapsto z+1$ still exists, where the conjugating maps are known as the Douady-Fatou coordinates. More presecisely, adapting \cite[Proposition 3.2.2]{shis2000}, \cite[Theorem 2.1]{lei2000} to our setting we give the following.

\begin{figure}[htb!]
\begin{center}
\def\svgwidth{6cm}
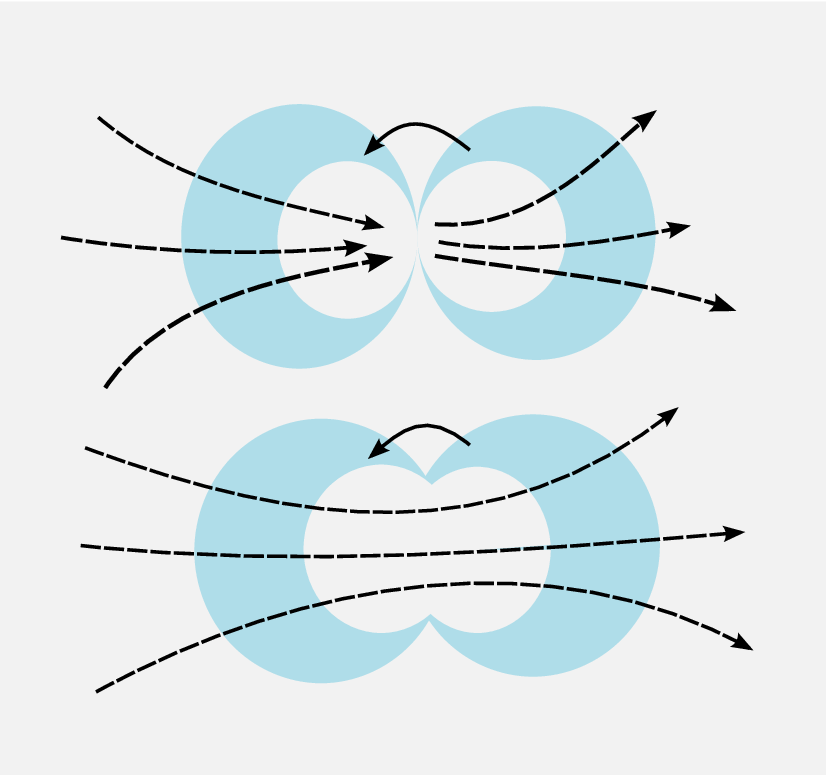
\caption{\small{Bifurcation of the parabolic fixed point. }}\label{eggbeatereps}
\end{center}
\end{figure}

\begin{thm}[Existence of Douady-Fatou coordinates ]\label{thmdouadyfatoucoordinate}
 Consider the local natural family with parabolic assumptions $\{f_a\}_{a\in\mathbb{D}(a_0,r)}$. There exists an open sector $\Delta\subset\mathbb{D}(a_0,r)$ with corner point $a_0$, and there exists a pair of conformal maps $z\mapsto\phi_a^{\pm}(z)$ for all $a\in\Delta$, unique additive to a constant, which satisfy
 \begin{equation*}
\phi_{a}^{+}(f_{a}(z))=\phi_{a}^{+}(z)+1\;\;\mathrm{and}\;\;\phi_{a}^{-}(f_{a}(z))=\phi_{a}^{-}(z)+1.
\end{equation*}
  After suitable normalization, $\phi_a^{\pm}$ depend holomorphically on $a$ and  $\phi_a^{\pm}\rightrightarrows\phi_{a_0}^{\pm}$ locally uniformly on compact sets. Moreover, for a fixed parameter $a$, the difference $\phi_a^{-}(z)-\phi_a^{+}(z)$ is constant.
\end{thm}
In fact, the theory of parabolic implosion was originally constructed for the family of rational maps. However, since it is a local theory, organizing local conditions, it can apply to other families. In our setting, the key idea is choosing a repelling petal $\Omega^{-}$ as a subset of $\overline{\mathbb{D}(z_0,R)}$ (recall that $\Omega^{+}$ has only $s(a_0)$, by assumption). This guarantees only $s(a)$ of all singular values to be involved in egg beater dynamics so that the theory applies. In other words, regardless of the parameter, having only one singular value in $\overline{\mathbb{D}(z_0,R)}$ satisfies the necessary local conditions of the parabolic impolosion theory to apply.

\section{Reparametrization}\label{mu}
By Theorem~\ref{thmdouadyfatoucoordinate}, the difference $\phi_a^{-}(z)-\phi_a^{+}(z)$ between the Douady-Fatou coordinates is constant for a fixed parameter. Define
\begin{equation}\label{B}
B:=B(a)=\phi_a^{-}-\phi_a^{+}.
\end{equation}
Since after suitable normalizations, $\phi_a^{+}$ and $\phi_a^{-}$ depend holomorphically on the parameter on some open sector, $B$ is also holomorphic on the same sector. In this section we present the following result. 
\begin{prop}\label{parameterB}
There exists an open sector $\Delta\subset\mathbb{D}(a_0, r)$ with the corner point $a_0$ such that $B$ depends on the parameter  $a$ univalently. 
\end{prop}

Let  $z_1(a)$ and $z_2(a)$ be the two fixed points bifurcated from the parabolic fixed point $z_0$ when $f_{a_0}$ is perturbed to $f_a$ nearby. These fixed points might be equal. Possibly taking $r'<r$, we can suppose that $z_1(a)\neq z_2(a)$ for $a\in\mathbb{D}^{*}(a_0,r')$.  Because the family  depends holomorphically on the parameter, the multipliers of these two fixed points also depend holomorphically on the parameter. By an affine change of coordinates in the dynamical plane, we translate one of the fixed points to the origin, and hence locally, we express the map as
\begin{equation}\label{mu(a)}
f_a(z)=\mu(a)z+O(z^2).
\end{equation}

We are interested in the multiplier map $a\mapsto\mu(a)$.
\begin{prop}\label{univalentmultiplier}
There exists an open sector $\Delta$ in the parameter plane with corner point $a_0$, such that the multiplier map $\mu:\Delta\rightarrow\mathbb{C}$ of the fixed point $0$ of $f_a$ in (\ref{mu(a)}) is univalent.
\end{prop}
The proof requires a change of coordinates: by a univalent map $\lambda:a\mapsto\lambda(a)$, we change the parameter and write (\ref{mu(a)}) as
\begin{equation}\label{normalizedformofthemap}
f_{a(\lambda)}(z)=(1+2\lambda^{n})z+z^{2}+...
\end{equation}
so that the multiplier of the fixed point $0$ is $\rho(\lambda):=\mu(a(\lambda))=1+2\lambda^{n}$ (Lemma \ref{newcoordinate}). The rest is only restricting the domain of $\rho$ so that it is a univalent function of $\lambda$, and hence $\mu$ is a univalent function of $a$. 

\setcounter{thm}{2}
\begin{lem}\label{lemsymmetricfixedpoints} Let $\mathbb{D}(z_0,R)$ be a Euclidean disk which contains only two fixed points $\{z_1(a),z_2(a)\}$ of $f_{a}$. There exists an affine change of coordinates in the dynamical plane, which depends holomorphically on the parameter such that the corresponding fixed points are symmetric with respect to  $0$.
\end{lem}
\begin{proof}
Define $\mathbb{S}(z_0,R):=\partial \mathbb{D}(z_0,R)$. As a result of the Generalized Argument Principle, the sum of the fixed points $z_1(a)$ and $z_2(a)$ is the holomorphic function given by
\begin{equation*}
a\mapsto\sigma(a)=\frac{1}{2\pi i}\oint_{\mathbb{S}(z_0,R)}\frac{w(f_a'(w)-1)}{f_a(w)-w}dw.
\end{equation*}
The linear map $M_a(z)=z-\frac{1}{2}\sigma(a)$ conjugates $f_a$ to the holomorphic map, which has two symmetric fixed points with respect to $0$.
\end{proof}

After the affine change of coordinates given in Lemma~\ref{lemsymmetricfixedpoints}, let us denote the fixed points by $z(a)$ and $-z(a)$ corresponding to $z_1(a)$ and $z_2(a)$. It is possible to find a local expression for them. This is given by the following lemma.
\begin{lem}\label{lemfixedpoints}
Consider a map that has two symmetric fixed points $\pm z(a)$. % as in Lemma~\ref{lemsymmetricfixedpoints}. 
Taking a covering of degree at most $2$ in the parameter plane, the fixed points are holomorphic functions of the parameter of the form:
\begin{equation}\label{symmetricfixedpoints}
a\mapsto \pm z(a)=\pm (a-a_0)^n+O\Big((a-a_0)^{n+1}\Big).
\end{equation}
\end{lem}

\begin{proof}
Since $f_a$ depends holomorphically on the parameter, possibly reducing $r$ to $r'$, we may follow the fixed point $z(a)$ analytically along any path in $\mathbb{D}^*(a_0,r')$. Consider a simple closed curve $\gamma:[0,1]\rightarrow\mathbb{C}$ around $a_0$ in $\mathbb{D}^*(a_0,r')$ in the parameter plane. For the parameters $\gamma(0)$ and $\gamma(1)$, we observe two different situations in the corresponding dynamical planes. Either
\begin{enumerate}
\item[(i)] the locations of the fixed points $z(a)$ and $-z(a)$ are the same for $\gamma(0)$ and $\gamma(1)$, or
\item[(ii)] they have interchanged their locations.
\end{enumerate}
In case (i), $z(a)$ is a holomorphic function of $a\in\mathbb{D}^*(a_0,r')$ of local degree $n\geq 1$. Then, (\ref{symmetricfixedpoints}) is obtained after scaling. In case (ii), we consider a branched covering map of degree $2$ from a domain $\widehat{\Omega}$ to $\mathbb{D}(a_0,r')$ in the parameter plane (e.g., $b\mapsto a(b)=a_0+b^2$),
\begin{eqnarray*}
 \xi:\widehat{\Omega}&\rightarrow&\mathbb{C},\notag\\
 b&\mapsto&\xi(b)=a,\;\;\xi(b_0)=a_0.
 \end{eqnarray*}
Hence, the map $b\mapsto z(\xi(b))$ is a branched covering of degree $n\geq 1$. In this new paramerization, we are back in (i). Thus, the fixed points have the following form:
\begin{equation*}
b\mapsto\pm(b-b_0)^n+O((b-b_0)^{n+1})
\end{equation*}
after scaling.
\end{proof}

In the following lemma, we find the multiplier of the fixed point $z(a)$ given in (\ref{symmetricfixedpoints}).
\begin{lem}\label{newcoordinate} Possibly changing the coordinate in the parameter plane, the multiplier of the fixed point $z(a)$ can be given by
\begin{equation}\label{multipliernewparameter}
\lambda\mapsto\rho(\lambda):=1+2\lambda^n
\end{equation}
with $\lambda_0=0$. 
\end{lem}

\begin{proof}

Near $0$, $f_a$ can be written as
\begin{equation*}
f_a(z)=z+(z^2-z(a)^2)h_{a}(z)
\end{equation*}
where $h_{a}:\mathbb{D}(z_0,R)\rightarrow\mathbb{C}$ is a nonvanishing holomorphic map. The multiplier $\mu$ of $z(a)$ is equal to
\begin{equation*}
\mu(a):=f_a'(z(a))=1+ 2 z(a)h_{a}(z(a)).
\end{equation*}
Substituting $z(a)=(a-a_0)^n+O((a-a_0)^{n+1})$ (see Lemma~\ref{lemfixedpoints}), we have
\begin{eqnarray*}
\mu(a)&=&1+2[(a-a_0)^n+O((a-a_0)^{n+1})]h_a(z(a))\\
&=&1+2(a-a_0)^n[1+O(a-a_0)]h_a(z(a))
\end{eqnarray*}
Define the holomorphic map $F(a):=[1+O(a-a_0)]h_a(z(a))$. Observe that $F(a_0)\neq 0$ since $h_a(z(a))$ is nonvanishing. Hence, $F$ has a well-defined $n$th root around $a_0$. Therefore, we can write
\begin{eqnarray}\label{multipliermap}
\mu(a)&=&1+2(a-a_0)^n\big((F(a))^{1/n}\big)^n\notag\\
&=&1+2\big((a-a_0) (F(a))^{1/n}\big)^n.
\end{eqnarray}
We define $\lambda(a):=(a-a_0) (F(a))^{1/n}$. Since $F(a_0)\neq 0$, then $\lambda'(a_0)\neq 0$, and hence there exists a neighbourhood of $a_0$ where the map $a\mapsto\lambda(a)$ is univalent. Thus, we can consider $\lambda$ as a new parameter. Rewriting (\ref{multipliermap}) in the new parameter, the multiplier $\mu(a)$ of $z(a)$ is written as $\rho(\lambda):=1+2\lambda^n$, a holomorphic function of $\lambda$ with $\lambda(a_0)=0$. 
\end{proof}

\begin{proof}[Proof of Proposition \ref{univalentmultiplier}]
With the result of Lemma \ref{newcoordinate}, shifting $z(a)$ to $0$, consider $f_{a(\lambda)}$ expressed as
 in (\ref{normalizedformofthemap}). % (compare (\ref{mu(a)})).
In the $\lambda$-parameter plane, consider the circular sector
\begin{equation*}\label{lambdasector}
\Lambda=\{\lambda:\;\;2\pi\theta_1<\arg(\lambda)<2\pi\theta_2,\;\;\theta_2, \theta_2\in(0,1)\}.
\end{equation*}
Then, $\arg(\rho(\lambda)-1)\in(2\pi\theta_1, 2\pi\theta_2)$. Provided that $n(\theta_2-\theta_1)<1$, the restriction of $\rho$ on $\Lambda$ is a univalent function of $\lambda$. This means, the restriction of $\mu$ on $\Delta=\lambda^{-1}(\Lambda)$ is also a univalent function of $a$. 

\end{proof}
\begin{figure}[htb!]
\begin{center}
\def\svgwidth{10cm}
\includegraphics[scale=.7]{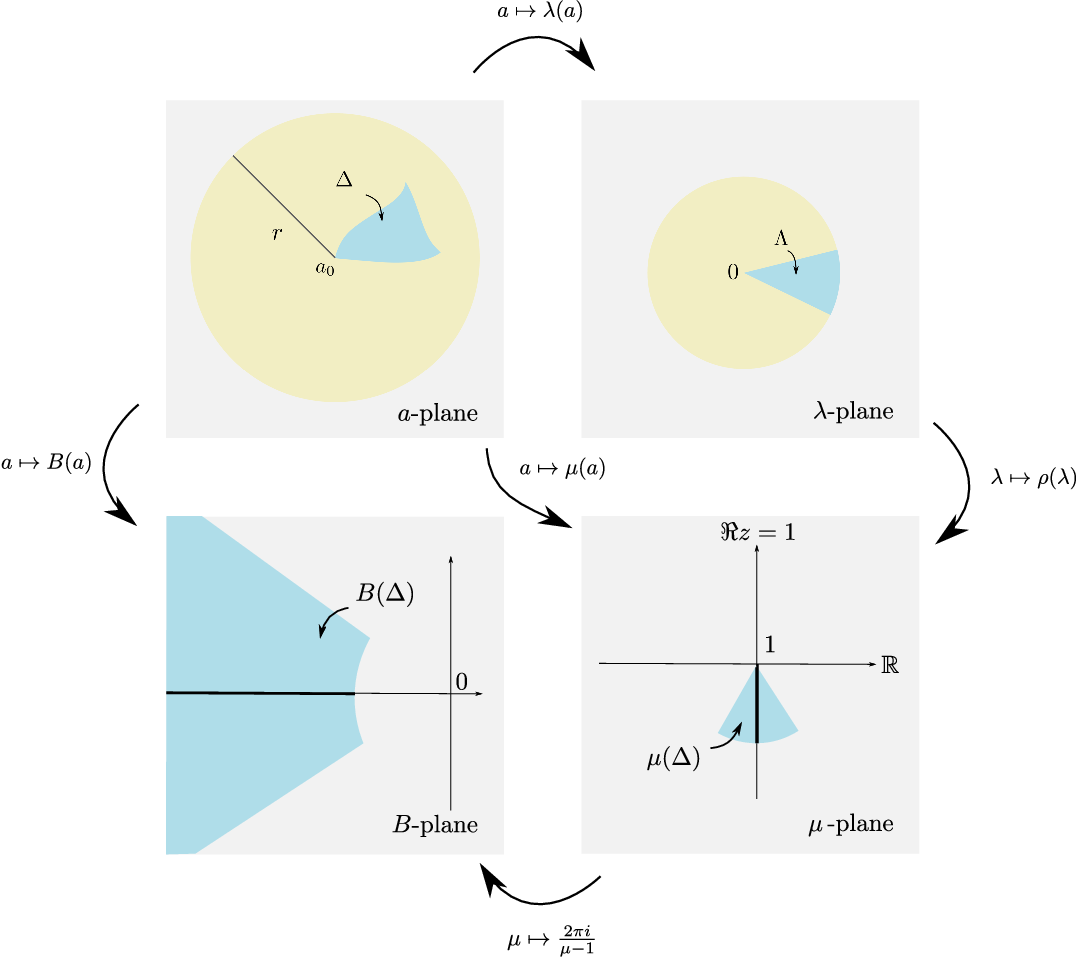}
\caption{\small{Correspondence of the regions in the parameter planes and the multiplier. Observe that $\mu(\Delta)=\rho(\Lambda))$ in $\mu$-plane.}}
\label{abplane}
\end{center}
\end{figure}
The following is a result of Douady-Lavaurs-Shishikura, expressed in our setting.

\begin{thm}\label{propositionmultiplierb}
Under the hypothesis of Theorem~\ref{thmdouadyfatoucoordinate}, if $\Delta$ is chosen where the multiplier $a\mapsto\mu(a)$ of the fixed point $0$ of (\ref{mu(a)}) is univalent, then
\begin{equation}\label{Bmu}
B=B(a)=\phi_a^{-}-\phi_a^{+}=\frac{2\pi i}{\mu(a)-1}
\end{equation}
holds after suitable normalizations of $\phi_a^{\pm}$.
\end{thm}

For the proof, reader can see \cite[Theorem 2.1]{lei2000}, or \cite{ADeniz2014} alternatively.

\begin{proof}[Proof of Proposition \ref{parameterB}] There exists an open sector $\Delta$ in $a$-plane where the multiplier map $\mu$ is univalent by Proposition~\ref{univalentmultiplier}. By (\ref{Bmu})  in Theorem~\ref{propositionmultiplierb}, $B$ is also univalent with respect to $a$ in $\Delta$. 
\end{proof}

\begin{rem}
In order to match our work with the classical description of the parabolic implosion, we choose a circular sector $\Lambda$ in $\lambda$-plane so that for the circular sector $\rho(\Lambda)$, the Douady-Fatou coordinates exists. Then $\Delta=\lambda^{-1}(\Lambda)$ is a ``deformed'' sector. Note that by $B$, the line segment $\mu(\Delta)\cap\{z:\;\;\Re(z)=1,\;\;\Im(z)<0\}$ maps to a half line on the negative real axis in $B$-plane. Observe that $\mu(\Delta)$ are $\rho(\Lambda)$ same (see Figure~\ref{abplane}).
\end{rem}

\section{Proof of the Main Theorem}\label{sectionproof}

Consider the forward invariant curve $\gamma_{a_0}$ which lands at the parabolic fixed point $z_0$ of $f_{a_0}$. Here $\gamma_{a_0}$ lands through the repelling petal $\Omega^{-}$ as any orbit under $f_{a_0}^{-1}$ converging to $z_0$ asymptotically converges to the repulsion vector (see \cite[Lemma 10.1]{Mil}). This means $\gamma_{a_0}$ must be outside the immediate parabolic basin. Note that, depending on $f_{a_0}$, $\gamma_{a_0}$ can be contained in a Fatou component having $z_0$ on its boundary, or in the Julia set.
The Fatou coordinate $\phi_{a_0}^{-}$ maps $\gamma_{a_0}\cap\Omega^{-}$ to a $1$-periodic curve in the Fatou coordinate plane. We extend this curve $1$-periodically towards $-\infty$ and $+\infty$, and denote the extension by $\widetilde{\gamma}_{a_0}$. 
\begin{defn}[Central sector]
If $\Delta$ is chosen so that $\mu(\Delta)\cap\{z:\;\Re(z)=1,\;\Im(z)<0\}\neq\emptyset$, we call it \textit{a central sector}. 
\end{defn}

Consider a central sector $\Delta$ with corner point $a_0$, on which the Douady-Fatou coordinates are well defined, and moreover, $a\mapsto B(a)=\phi_a^{-}-\phi_a^{+}$ is univalent (see Section \ref{mu}). Recall that  $\phi_a^{-}\rightrightarrows\phi_{a_0}^{-}$ on compact subsets of the repelling petal $\Omega^{-}$. Observe that because $\Delta$ is central, the region $B(\Delta)$ is unbounded from left, by (\ref{Bmu}).

Denote by $\gamma_{a_{0}}[t,t+1]$, the fundamental segment with endpoints $\gamma_{a_0}(t)$ and $\gamma_{a_0}(t+1)$ on the forward invariant curve $\gamma_{a_0}$.

Let $H:\mathbb{D}(a_0,\delta)\times\gamma_{a_0}[T,\mathcal{T})\rightarrow\widehat{\mathbb{C}}$ be the equivariant holomorphic motion, and $\gamma_a(t):=H(a,\gamma_{a_{0}}(t))$. By the following three steps, we may need to restrict the parameter set $\mathbb{D}(a_0,\delta)$. 

\begin{itemize}
    \item[1.] Given $T'<T$, we may assume that $H$ extends equivariantly to $\mathbb{D}(a_0,\delta_1)\times\gamma_{a_0}[T',\infty)$, for some $\delta_1\leq \delta$, by using the dynamics.
\end{itemize} Suppose $t'\in\mathbb{R}$, such that $\gamma_{a_0}[t',t'+2]\subset\Omega^{-}$,  $H:\mathbb{D}(a_0,\delta_1)\times\gamma_{a_0}[t',\mathcal{T})\rightarrow\mathbb{C}$ is an equivariant extension. 

\begin{itemize}
    \item[2.] Possibly reducing $\delta_1$ to $\delta_2$, we can assume that
\begin{equation*}
K:=H(\overline{\mathbb{D}(a_0,\delta_2)},\gamma_{a_0}[t',t'+2])\subset\Omega^{-},
\end{equation*}
by the continuity of the holomorphic motion.
Here $K\subset\Omega^{-}$ consists of curves each of which is in overlapping dynamical planes.

\item[3.] By a further reduction of $\delta_2$ to $0<\delta_3\leq \delta_2$, we can assume that $K$ is contained in the domain of the outgoing Douady-Fatou coordinate $\phi_a^{-}$ for $a\in\Delta_3:=\Delta\cap\overline{\mathbb{D}(a_0,\delta_3)}$, where $\Delta$ is the initial central sector.
\end{itemize}

 For $a\in\Delta_3$, $\phi_a^{-}$ sends $\gamma_a[t',t'+2]$ to $\phi_a^{-}(\gamma_a[t',t'+2])$, which extends $1$-periodically towards $-\infty$ and $+\infty$, and denote the extension by $\widetilde{\gamma}_a$ (see Figure~\ref{overlappedrays}). 

\begin{figure}[htb!]
\begin{center}
\includegraphics[scale=.2]{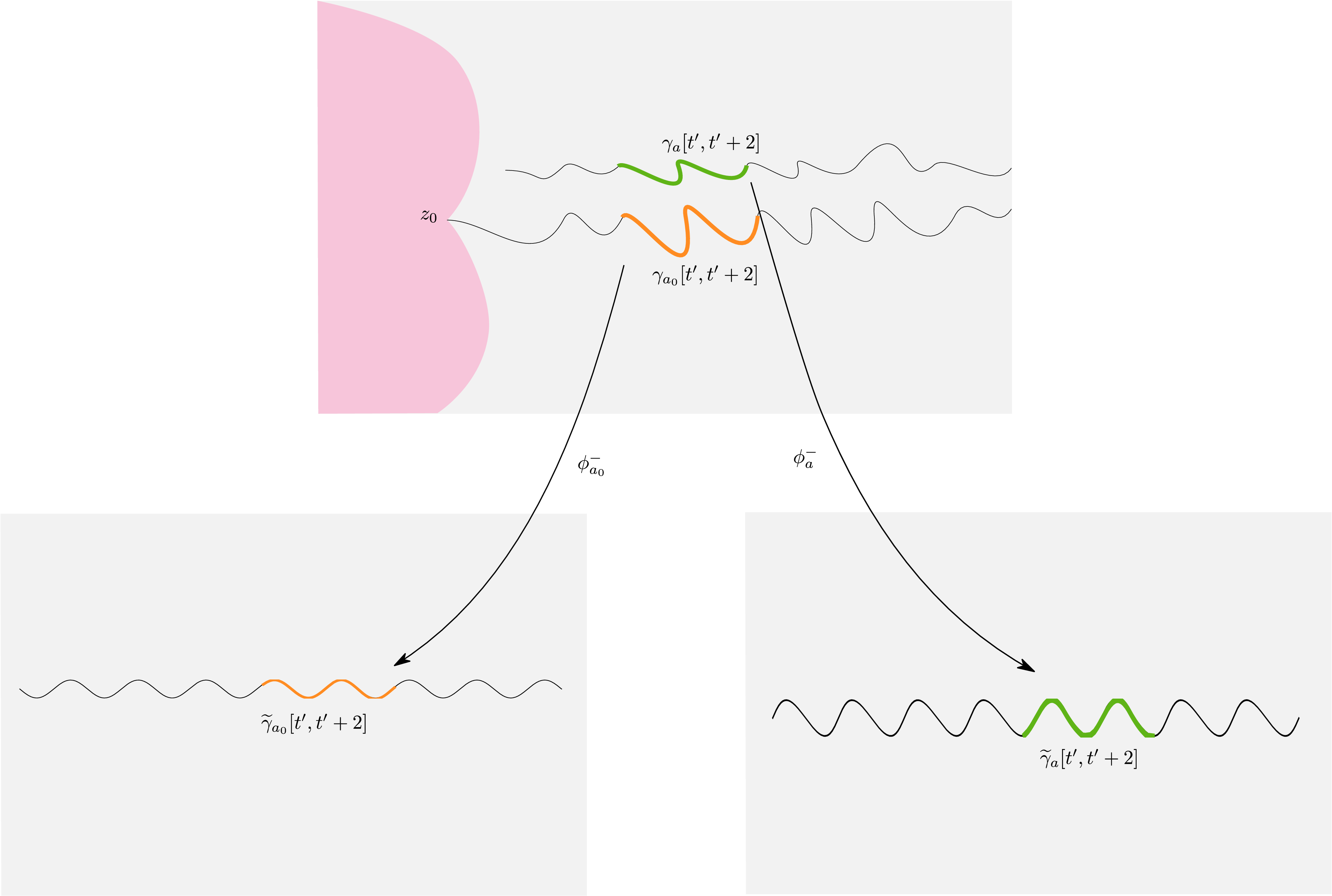}
\end{center}
\caption{\small{Projection of curve pieces in $K$ in the (Douady-) Fatou Coordinate planes and their $1$-periodic extensions.} }\label{overlappedrays}
\end{figure}

Denote by $\widetilde{\gamma}_{a}[t,t+1]=\phi_a^{-}({\gamma}_a[t,t+1])$, the fundamental segment with endpoints $\widetilde{\gamma}_a(t)$ and $\widetilde{\gamma}_a(t+1)$ on the curve $\widetilde{\gamma}_a$, and by $[\widetilde{\gamma}_a(t),\widetilde{\gamma}_a(t+1)]$ the line segment with endpoints $\widetilde{\gamma}_a(t)$ and $\widetilde{\gamma}_a(t+1)$.
\begin{lem}\label{lemisotopy}
There exists an isotopy of curves
\begin{eqnarray*}
\mathcal{I}:\overline{\Delta_3}\times[t',t'+1]\times[0,1]\times[0,1]&\rightarrow&\mathbb{C}\\
(a,t'',s,t)&\mapsto&\mathcal{I}(a,t'',s,t)
\end{eqnarray*}
such that
\begin{eqnarray*}
\mathcal{I}(a,t'' ,0,[0,1])&=&\widetilde{\gamma}_{a}[t'',t''+1],\\
\mathcal{I}(a,t'',1,[0,1])&=&[\widetilde{\gamma}_{a}(t''),\widetilde{\gamma}_{a}(t''+1)],
\end{eqnarray*}
relative to the endpoints $\widetilde{\gamma}_{a}(t'')$ and $\widetilde{\gamma}_{a}(t''+1)$. Moreover, 
\begin{enumerate}
    \item [i.] $t''\mapsto\mathcal{I}(a,t'',s,t)$ is continuous,
    \item [ii.] $a\mapsto\mathcal{I}(a,t'',s,t)$ is continuous,
    \item [iii.] $t\mapsto\mathcal{I}(a,t'',s,t)$ is $T_1$ invariant.
\end{enumerate}

\end{lem}

\begin{proof}
Take $y$ large enough so that the horizontal line $\mathcal{L}:=\mathbb{R}+iy$ is above $\widetilde{\gamma}_a$ and so $\mathcal{L}\cap\widetilde{\gamma}_{a}=\emptyset$. %(take $y$ large enough so that $\mathcal{L}$ does not depend on the parameter).
Denote by $\chi_a$ the domain bounded by $\mathcal{L}$ from above and by $\widetilde{\gamma}_a$ from below. Since $\chi_a$ is simply connected, by the Uniformization Theorem, there exists a conformal map $\pi_a:S\rightarrow \chi_a$ from the horizontal strip $S:=\{z:\;\;0<\Im z<1\}$ to $\chi_a$. Moreover, since $\partial\chi_a$ is a Jordan curve, $\pi_a$ extends continuously to $\partial S$, say $\pi(\mathbb{R})=\widetilde{\gamma}_a$ and $\pi(\mathbb{R}+i)=\mathcal{L}$. The composition $\pi_a^{-1}\circ T_1\circ\pi_a$ defines an isomorphism on $S$, which is necessarily a translation $T_c(z):=z+c$ for some $c\in\mathbb{R}^{+}$. 

\begin{figure}[htb!]
\begin{center}
\includegraphics[scale=.5]{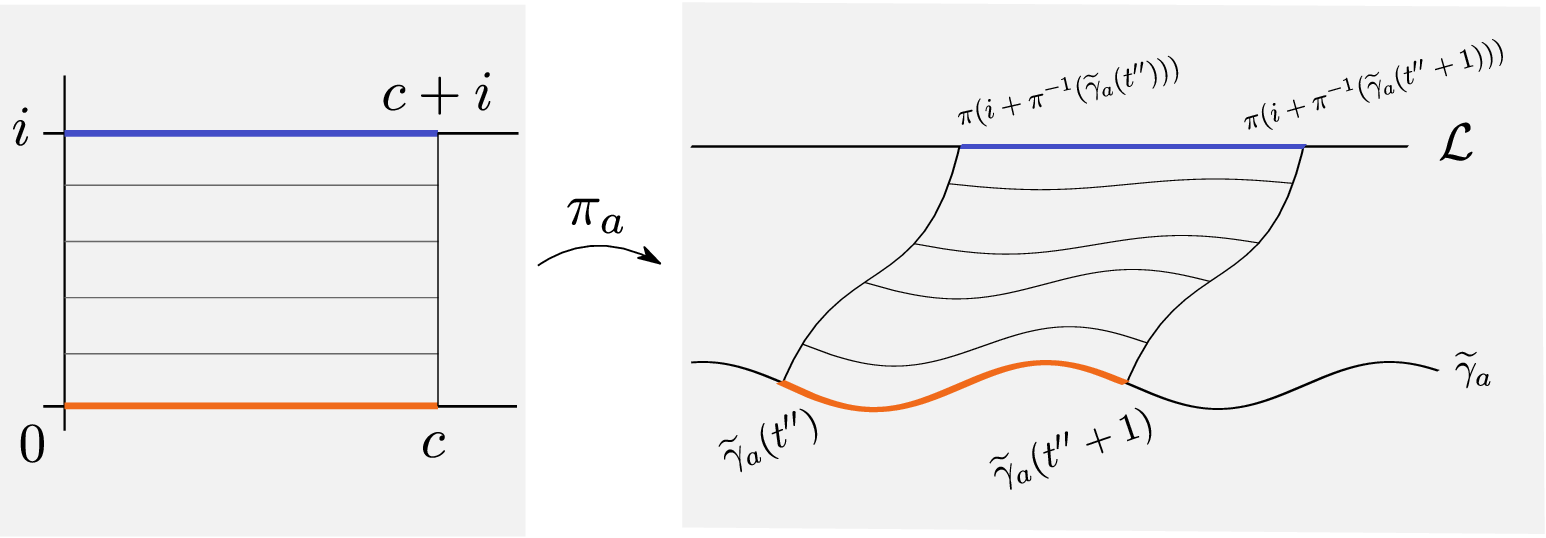}
\end{center}
\caption{\small{Conformal map $\pi_a:S\rightarrow \chi_a$, normalized so that $\pi_a(0)=\widetilde{\gamma}_a(t'')$ for a simpler demonstration.}}\label{sertac1}
\end{figure}

We define the isotopy of curves as follows (see Figure \ref{sertac1}):
\begin{equation*}
\mathcal{I}(a,t'',s,t):=\pi_a\Big(is+\pi_a^{-1}(\widetilde{\gamma}_a(t''+t))\Big)+\widetilde{\gamma}_a(t'')-\pi_a\Big(is+\pi_a^{-1}(\widetilde{\gamma}_a(t''))\Big).
\end{equation*}

Observe that $\mathcal{I}(a,t'',0,t)=\widetilde{\gamma}_a(t''+t)$ and

\begin{equation*}
    \mathcal{I}(a,t'',0,[0,1])=\widetilde{\gamma}_a[t'', t''+1].
\end{equation*}

Observe also
\begin{equation*}
\pi_a(i+\pi_a^{-1}\left(\widetilde{\gamma}_a(t''+t)))-\pi_a(i+\pi_a^{-1}(\widetilde{\gamma}_a(t''))\right)=t,
\end{equation*}
and hence we have 
\begin{equation*}
    \mathcal{I}(a,t'',1,[0,t])=\widetilde{\gamma}_a(t'')+[0,t]=[\widetilde{\gamma}_a(t''), \widetilde{\gamma}_a(t''+t)],
\end{equation*}
and in particular 
\begin{equation*}
    \mathcal{I}(a,t'',1,[0,1])=\widetilde{\gamma}_a(t'')+[0,1]=[\widetilde{\gamma}_a(t''), \widetilde{\gamma}_a(t''+1)].
\end{equation*}

\begin{itemize}
    \item[i.] Since $\pi_a$ is conformal and $\widetilde{\gamma}_a$ is a continuous curve, $\mathcal{I}$ is continuous with respect to $t''$.
    \item[ii.] Since $\gamma_a[t',t'+2]\subset K$,  $a\mapsto\gamma_a[t',t'+2]$ is continuous by the holomorphic motion. Because $\gamma_a[t',t'+2]$ is also contained in the domain of the outgoing Douady-Fatou coordinate $\phi_a^{-}$, and $\phi_a^{-}$ depends continuously on the parameter, $a\mapsto\widetilde{\gamma}_a[t', t'+2]=\phi_a^{-}(\gamma_a[t', t'+2])$ is also continuous. And since $a\mapsto \pi_a$ is continuous, we conclude that $\mathcal{I}$ is continuous with respect to $a\in\Delta_3$.
 
 \item[iii.] Using $\pi_a\circ T_c=T_1\circ\pi_a$, observe
\begin{eqnarray*}
\pi_a\left(is+\pi_a^{-1}(\widetilde{\gamma}_a(t''+t+1))\right)&=&\pi_a\left(is+\pi_a^{-1}(\widetilde{\gamma}_a(t''+t)+1)\right)\\
&=&\pi_a\left(is+\pi_a^{-1}(\widetilde{\gamma}_a(t''+t))+c\right)\\
&=&\pi_a\left(is+\pi_a^{-1}(\widetilde{\gamma}_a(t''+t)\right)+1,
\end{eqnarray*}

and similarly
\begin{equation*}
\pi_a\left(is+\pi_a^{-1}(\widetilde{\gamma}_a(t''))\right)=\pi_a\left(is+\pi_a^{-1}(\widetilde{\gamma}_a(t'')\right)+1,
\end{equation*}

and since $\widetilde{\gamma}_a(t''+1)=\widetilde{\gamma}_a(t'')+1$,
we obtain

\begin{equation*}
\mathcal{I}(a,t'',s,t)-\mathcal{I}(a,t'',s,t-1)=1,
\end{equation*}

that is, $\mathcal{I}$ is $T_1$-equivariant in the variable $t$.
\end{itemize}
\end{proof}
\begin{proof}[Proof of the Main Theorem -- Conclusion.]
Since $\overline{\Delta_3}$, $[t',t'+1]$, $[0,1]$ are compact, and $\mathcal{I}$ is continuous in $(a,t'',s,t)$, then  $\widetilde{C}=\mathcal{I}(\overline{\Delta_3}\times[t',t'+1]\times[0,1]\times[0,1])$ is also compact. We may need extra restriction in order for $\widetilde{C}$ to be included in $\phi_{a_0}(\Omega^{-})$ and in the domain of $\psi_a^{-}=(\phi_a^{-})^{-1}$ as follows:

\begin{itemize}
    \item[1.] If $\widetilde{C}$ is not contained in $\phi_{a_0}^{-}(\Omega^{-})$, possibly reducing $\delta_3$ to $\delta_4$, and using the $T_1$-equivariance of $\mathcal{I}$,

we can assume $\widetilde{C}-n\subset\phi_{a_0}^{-}(\Omega^{-})$ for some $n\in\mathbb{N}$. 
\item[2.] Finally, if $\widetilde{C}-n$ is not in the domain of $\psi_a^{-}=(\phi_a^{-})^{-1}$ for $a\in\Delta_4:=\Delta\cap\mathbb{D}(a_0,\delta_4)$, with a further reduction $\delta_5$ from $\delta_4$, we can assume that $\widetilde{C}-n$ is in the domain of $\psi_a^{-}$ for  $a\in\Delta_5=\Delta\cap\mathbb{D}(a_0,\delta_5)$. 
\end{itemize}

 Let $\widetilde{C}_{a,t''}$ denote the restriction $\mathcal{I}|_{\{a\}\times\{t''\}\times[0,1]\times[0,1]}$ for $t''\in[t'-n,t'-n+1]$ and $a\in\Delta_5$. Consider the equivariant extension $H:\mathbb{D}(a_0,\delta_5)\times\gamma_{a_0}[t'-n,\mathcal{T})$.

For $a\in\overline{\Delta_5}$, let $\widetilde{\gamma}_{a_0}$ and $\widetilde{\gamma}_a$ denote the extensions of $\phi_{a_0}^{-}(\gamma_{a_0}[t',t'+2])$ and $\phi_{a}^{-}(\gamma_{a}[t',t'+2])$, respectively. We use the new parameter $B$, as $a=a(B)$ (Proposition~\ref{parameterB}). To fix normalization, in addition, we shall suppose that $B:=\phi_a^{-}(s(a))$ (then by this choice of normalization, $\phi_a^{+}(s(a))=0$, by (\ref{B})).

By the continuity of $\widetilde{\gamma}_a(t'')$ on the parameter at $a_0$, we can define
\begin{equation}\label{supepsilon}
\epsilon:=\sup_{t''\in[t'-n,t'-n+1],a\in\overline{\Delta_5}}|\widetilde{\gamma}_a(t'')-\widetilde{\gamma}_{a_0}(t'')|.
\end{equation}

Because $\Delta_5$ is a central sector, the region $B(\Delta_5)$ is unbounded from the left, and contains a half line on the negative real axis. Since $\widetilde{\gamma}_{a_0}$ is $1$-periodic, there always exists $t_0$ such that for all $t\leq t_0$,  $\widetilde{\gamma}_{a_0}(t)$ stays inside $B(\Delta_5)$ and away from the boundary of $B(\Delta_5)$. Choose $t_0$ with the property that, for all $t\leq t_0$:

\begin{equation*}
\widetilde{\gamma}_{a_0}(t)\in B(\Delta_5)\;\;\;\;\;\mathrm{and}\;\;\;|\widetilde{\gamma}_{a_0}(t)-\partial B(\Delta_5)|>\epsilon.
\end{equation*}

We want to compare $\widetilde{\gamma}_a(t)$ and $B$. Consider the following identity:
\begin{equation}\label{starequality}
\widetilde{\gamma}_a(t)-B=\Big(\widetilde{\gamma}_a(t)-\widetilde{\gamma}_{a_0}(t)\Big)+\Big(\widetilde{\gamma}_{a_0}(t)-B\Big).
\end{equation}
First, we will show that whenever $a\in\Delta_5$, then 
\begin{equation}\label{continuityinequality5}
|\widetilde{\gamma}_{a}(t)-\widetilde{\gamma}_{a_0}(t)|<\epsilon,\;\;\;\;\;\text{for all}\;\;\; t\leq t_0.
\end{equation}

Let $k$ be the natural number such that $t+k:=t''\in[t'-n,t'-n+1]$. By $1$-periodicity,
\begin{equation*}
\widetilde{\gamma}_a(t)+k=\widetilde{\gamma}_{a}(t+k)=\widetilde{\gamma}_a(t''),
\end{equation*}
and so
\begin{equation*}
\widetilde{\gamma}_{a}(t)-\widetilde{\gamma}_{a_0}(t)=\widetilde{\gamma}_{a}(t'')-\widetilde{\gamma}_{a_0}(t'').
\end{equation*}
Therefore, (\ref{supepsilon}) implies (\ref{continuityinequality5}).

Set $B_0:=\widetilde{\gamma}_{a_0}(t)$, and define holomorphic functions of $B$ in $B(\Delta_5)$,
\begin{eqnarray*}
\xi_1(B)&:=&\widetilde{\gamma}_{a(B)}(t)-B\;\;\;\;\mathrm{and}\notag\\
\xi_2(B)&:=&B_0-B.
\end{eqnarray*}
The map $\xi_1$ is a holomorphic function of $B$ because the Douady-Fatou coordinates and the singular value $s(a)$ depend holomorphically on the parameter $a$ in $\Delta_5$ and hence on the parameter $B$ in $B(\Delta_5)$. Obviously, $\xi_2$ is a holomorphic map of $B$.
In terms of $\xi_1$ and $\xi_2$, (\ref{starequality}) can be written as
\begin{equation*}
\xi_1(B)=\Big(\widetilde{\gamma}_{a}(t)-\widetilde{\gamma}_{a_0}(t)\Big)+\xi_2(B).
\end{equation*}
This yields, by (\ref{continuityinequality5}),
\begin{equation*}
|\xi_1(B)-\xi_2(B)|<\epsilon\;\;\;\;\mathrm{for}\;\;B\in B(\Delta_5).
\end{equation*}
For $B\in \mathbb{S}(B_0,\epsilon):=\partial \mathbb{D}(B_0,\epsilon)$, $|\xi_2(B)|=\epsilon$. Since $\xi_2$ has one simple zero at $B_0$ in $\mathbb{D}(B_0,\epsilon)$, then $\xi_1$ has one simple zero in $\mathbb{D}(B_0,\epsilon)$ by Rouch\'e's Theorem (see Figure~\ref{sunuhade1}). This means that in $\mathbb{D}(B_0,\epsilon)$, there exists a unique parameter $B'$, which satisfies
\begin{equation*}
B'=\widetilde{\gamma}_{a(B')}(t),
 \end{equation*}
 equivalently
\begin{equation*}
 \phi_{a(B')}^{-}\Big(s(a(B'))\Big)=\widetilde{\gamma}_{a(B')}(t).
 \end{equation*}
Set $a':=a(B')$, then $\phi_{a'}^{-}\Big(s(a')\Big)=\widetilde{\gamma}_{a'}(t).$ We claim that $\gamma_{a'}$ has a unique $f_{a'}$-invariant extension until the singular value $s(a')$ and that
\begin{equation*}
s(a')=\gamma_{a'}(t).
\end{equation*}

\begin{figure}[htb!]
\begin{center}
\includegraphics[scale=.49]{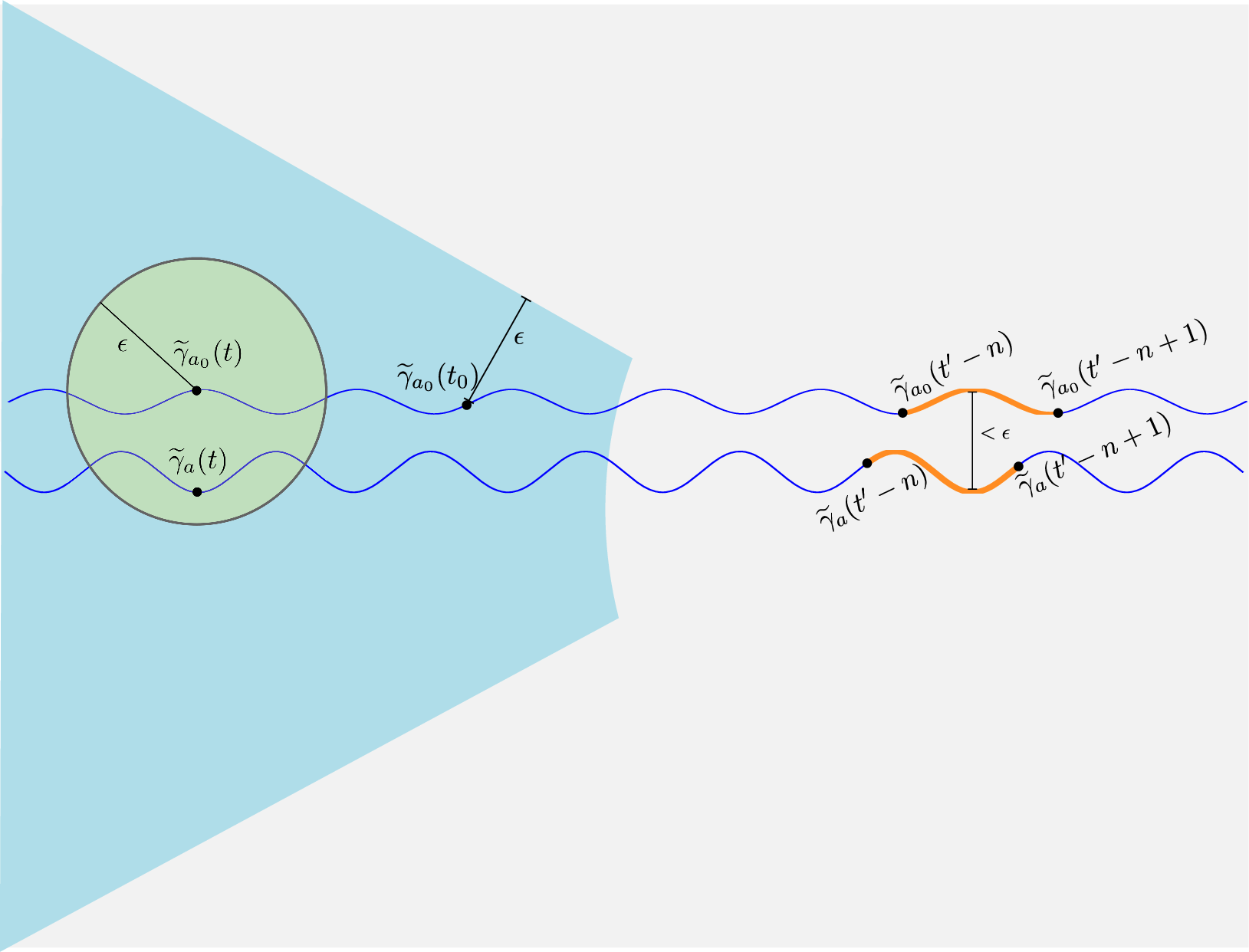}
\end{center}
\caption{\small{Demonstration of the proof of the Main Theorem. Blue region is $B(\Delta)$.}}\label{sunuhade1}
\end{figure}

\noindent We use the fact that the fundamental segment $\widetilde{\gamma}_{a'}[t'',t''+1]$ and the line segment $[\widetilde{\gamma}_{a'}(t''),\widetilde{\gamma}_{a'}(t''+1)]$ are isotopic relative to endpoints  $\widetilde{\gamma}_{a'}(t'')$ and $\widetilde{\gamma}_{a'}(t''+1)$. (See Lemma~\ref{lemisotopy}).
Moreover, the line segment $\widetilde{l}=[\phi_{a'}^{-}(s(a')),\widetilde{\gamma}_{a'}(t''+1)]$ is contained in the domain of the Douady-Fatou parameter $\psi_{a'}^{-}=(\phi_{a'}^{-})^{-1}$. Therefore, $l:=\psi_{a'}^{-}(\widetilde{l})$ is well defined and is a simple $f_{a'}$-invariant curve that connects $s(a')$ and $\gamma_{a'}(t'')$. On the other hand, the isotopy of curves $\widetilde{C}_{a',t''}$ maps to an isotopy of curves by  $\psi_{a'}^{-}$ in the dynamical plane for $f_{a'}$. By pulling back $\psi_{a'}^{-}([\widetilde{\gamma}_{a'}(t''),\widetilde{\gamma}_{a'}(t''+1)])$ $k$ times, we reach $s(a')$. More precisely, since $t+k=t''$, for the inverse branch $f^{-1}_{a'}$ along $l$, we have $f^{-k}_{a'}(\psi_{a'}(\widetilde{\gamma}_{a'}(t'')))=s(a')$.

Note that $\psi_{a'}^{-}([\widetilde{\gamma}_{a'}(t''),\widetilde{\gamma}_{a'}(t''+1)])$ is in the same isotopy class as $\psi_{a'}^{-}(\widetilde{\gamma}_{a'}[t'',t''+1])=\gamma_{a'}[t'',t''+1]$ relative to the set $\bigcup_{i=0}^{n+1}f_{a'}(s(a'))$. Therefore, by pulling back $\gamma_{a'}[t'',t''+1]$ under $f_{a'}$ $k$ times, we also reach $s(a')$. This gives us a well-defined extension of $\gamma_{a'}$ until  $s(a')$. In other words, the singular value $s(a')$ is on the forward invariant curve $\gamma_{a'}$ at potential $t$. 

This relation induces a map $\Gamma$, which assigns a unique parameter to each potential. More precisely, $\Gamma:(-\infty,t_0]\rightarrow\overline{\Delta_5}$, so that writing $a'=\Gamma(t)$ then $s(a')=\gamma_{a'}(t)$.

Now, we will show that $\Gamma$ forms a curve. By the Generalized Argument Principle, we have

\begin{equation*}
\xi^{-1}(w)=\frac{1}{2\pi i}\oint_{\mathbb{S}(B_0,\epsilon)}B\frac{\xi_1'(B)}{\xi_1(B)-w}dB
\end{equation*}
Knowing $\xi(B)=\widetilde{\gamma}_{a(B)}(t)-B$, and so $\xi^{-1}(\widetilde{\gamma}_{a(B)}(t)-B)=B$, when $B$ depends on $t$ in a way that $\widetilde{\gamma}_{a(B)}(t)=B$, then the integral in the following relation
\begin{equation*}
B(t)=\xi^{-1}(0)=\frac{1}{2\pi i}\oint_{\mathbb{S}(B_0,\epsilon)}B\frac{\xi_1'(B)}{\xi_1(B)}dB
\end{equation*}
is continuous, which means $t\mapsto B(t)$ is continuous.
Therefore, since $a(B(t))=\Gamma(t)$, then $\Gamma$ is a continuous curve, as it is parametrized by the continuous curve $(-\infty,t_0]$. 

Finally, we show that $\Gamma(t)$ lands at $a_0$. As $t\rightarrow -\infty$, $\gamma_{a_0}(t)\rightarrow z_0$, and  $\widetilde{\gamma}_a(t)\rightarrow-\infty$. This implies $\widetilde{\gamma}_a(t)=B\rightarrow-\infty$, and equivalently, $a\rightarrow a_0$. Hence, we obtain  as $t\rightarrow -\infty$,  $\Gamma(t)\rightarrow a_0$.
\end{proof}

\textbf{Acknowledgement:} This work is a part of my Ph.D. thesis. I would like to thank my
supervisor Carsten Lunde Petersen for introducing me to the problem, his guidance and inspiring discussions. Many thanks to my supervisor N\'{u}ria Fagella for her support and many valuable suggestions. Also, I would like to thank Anna Miriam Benini and Adam Epstein for pointing out the missing details. Special thanks to the referee for helping the work show its potential with many useful comments and corrections in both structure and mathematics. 
This work was supported by Marie Curie RTN 035651-CODY and Roskilde University, IMFUFA.

\end{document}

%% file: eggbeater_hop.eps_tex
%% Creator: Inkscape inkscape 0.48.1, www.inkscape.org
%% Accompanies image file 'eggbeater_hop.eps' (pdf, eps, ps)

\begingroup
  \makeatletter
  \providecommand\color[2][]{%
    \errmessage{(Inkscape) Color is used for the text in Inkscape, but the package 'color.sty' is not loaded}
    \renewcommand\color[2][]{}%
  }
  \providecommand\transparent[1]{%
    \errmessage{(Inkscape) Transparency is used (non-zero) for the text in Inkscape, but the package 'transparent.sty' is not loaded}
    \renewcommand\transparent[1]{}%
  }
  \providecommand\rotatebox[2]{#2}
  \ifx\svgwidth\undefined
    \setlength{\unitlength}{396.38708496pt}
  \else
    \setlength{\unitlength}{\svgwidth}
  \fi
  \global\let\svgwidth\undefined
  \makeatother
  \begin{picture}(1,0.93650794)%
    \put(0,0){\includegraphics[width=\unitlength]{eggbeater_hop.eps}}%
    \put(0.4867725,0.64550261){\color[rgb]{0,0,0}\makebox(0,0)[lb]{\smash{$0$}}}%
    \put(0.48148149,0.33597875){\color[rgb]{0,0,0}\makebox(0,0)[lb]{\smash{$z_1(a)$}}}%
    \put(0.48941804,0.19841273){\color[rgb]{0,0,0}\makebox(0,0)[lb]{\smash{$z_2$(a)}}}%
    \put(0.48148149,0.81481484){\color[rgb]{0,0,0}\makebox(0,0)[lb]{\smash{$f_{a_0}$}}}%
    \put(0.48941792,0.4444444){\color[rgb]{0,0,0}\makebox(0,0)[lb]{\smash{$f_a$}}}%
  \end{picture}%
\endgroup